\documentclass[11pt]{article}
\usepackage[pagebackref,colorlinks=true,urlcolor=blue,linkcolor=red,citecolor=red]{hyperref}
\usepackage{epsfig,graphics, graphicx}
\usepackage{subcaption, comment, bm}
\usepackage[percent]{overpic}
\usepackage{float}
\usepackage{amssymb,bm}
\usepackage{amsthm}
\usepackage{color}
\usepackage[]{amsmath}
\usepackage[]{amsfonts}
\usepackage[]{fancyhdr}
\usepackage[]{graphicx,wrapfig}
\usepackage{enumitem}
\usepackage[utf8]{inputenc}
\usepackage[T1]{fontenc}
\usepackage{mathtools}
\usepackage{array}

 \usepackage{pdfsync,verbatim,stmaryrd}
 \newcommand{\Beq}{\begin{equation}}
 \newcommand{\Eeq}{\end{equation}}
 \newcommand{\beq}{\begin{equation*}}
 \newcommand{\eeq}{\end{equation*}}
 \newcommand{\bal}{\begin{align}}
 \newcommand{\eal}{\end{align}}
 
 \newcommand{\D}{\mathrm{d}}

 \newcommand{\Lc}{\mathcal{L}}

 \newcommand{\Rc}{\mathcal{R}}
 \newcommand{\Sc}{\mathcal{S}}
 \newcommand{\Tc}{\mathcal{T}}

 \newcommand{\Rb}{\mathbb{R}}
 \newcommand{\Sb}{\mathbb{S}}

 \renewcommand{\o}{\omega}

 \usepackage{amsfonts}
 \newcommand{\FR}{\mathbb{R}} 

 \DeclareMathAlphabet{\bi}{OML}{cmm}{b}{it}
 \DeclareMathAlphabet{\bcal}{OMS}{cmsy}{b}{n}
 \DeclareMathAlphabet{\brmn}{OT1}{cmr}{bx}{n}
 \usepackage{amsopn}
 \usepackage{amsmath,amsthm,amssymb}

\newtheorem{remark}{Remark}
\newtheorem{proposition}{Proposition}

 \newtheorem{theorem}{Theorem}
 
 \newtheorem{lemma}{Lemma}

 \newtheorem{definition}{Definition}

\title{\vspace{-1cm} Inversion of generalized Radon transform over symmetric $m$-tensor fields in $\mathbb{R}^n$}

\author{Anuj Abhishek\thanks {Department of Mathematics, Applied Mathematics, and Statistics, Case Western Reserve University, United States. \url{anuj.abhishek@case.edu} }\and Rohit Kumar Mishra\thanks{Department of Mathematics, Indian Institute of Technology, Gandhinagar, Gujarat, India. \url{rohit.m@iitgn.ac.in}, \url{rohittifr2011@gmail.com}} \and Chandni Thakkar\thanks{Department of Mathematics, Indian Institute of Technology, Gandhinagar, Gujarat, India. \url{thakkar_chandni@iitgn.ac.in}}}

\begin{document}
\maketitle
\date{}
\begin{abstract}
In this work, we study a set of generalized Radon transforms over symmetric $m$-tensor fields in $\mathbb{R}^n$. The longitudinal/transversal Radon transform and corresponding weighted integral transforms for symmetric $m$-tensor field are introduced. We give the kernel descriptions for the longitudinal and transversal Radon transform. Further, we also prove that a symmetric $m$-tensor field can be recovered uniquely from certain combinations of these integral transforms of the unknown tensor field. This generalizes a recent study done for the recovery of vector fields from its weighted Radon transform data to recovery of a symmetric $m$-tensor field from analogously defined weighted Radon transforms.
\end{abstract}
\vspace{2mm}

\section{Introduction}\label{Sec:Introduction}
Radon transform was introduced as line integrals of functions in 2-dimensional space by Johann Radon \cite{Radon_transform} in 1917. The definition was then extended to higher dimensions with integrals of functions taken along hyperplanes. Since then, many different generalizations of the Radon transform and their inversions have been studied because of their wide range of applications, see \cite{Gaik_Book,Haltmeier_2011,Ren_2019,Kuchment_2006,Sharafutdinov_Sobolev_Radon,Webber_Quinto} and references therein. All the generalizations are primarily considered the integral of scalar functions over a variety of subsets (for instance, circles, ellipses, spheres, curves, V-lines, cones, and many other geometries). Recently, some extensions of the Radon transform acting on vector fields have been studied (\cite{Kunyansky_2023,Polyakova_Svetov_normal_Radon_numerical,Polyakova_normal_Radon_transform}) and to the best of authors' knowledge, there is only one work \cite{Polyakova_Svetov_Gen_Rad_tensor_fields} where the Radon transform over tensor fields in $\Rb^3$ is considered. In this work, we introduce a set of generalized Radon transforms acting on symmetric $m$-tensor fields in $\mathbb{R}^n$. In particular, we define longitudinal Radon transform (LRT) and transversal Radon transform (TRT) along with their weighted versions. Our aim here is to give a description of kernels of longitudinal and transversal Radon transforms and reconstruction of a symmetric $m$-tensor field from a combination of its LRT/TRT and weighted LRT/TRT. We note here that our definitions of LRT, TRT, and their weighted versions are natural generalizations of the well-known longitudinal/transverse ray transform, and the corresponding integral moments transforms, respectively. Similar transforms are studied in $\Rb^3$ for vector fields \cite{Kunyansky_2023}. 

The problem of recovering a tensor field in $\Rb^n$ using straight-line versions of the longitudinal, transverse, and integral moment (analogous to weighted versions in our setup) ray transforms has been considered by many research groups, and many references are available on the subject. It is well-known that only the solenoidal part of the unknown tensor field can be reconstructed from its longitudinal ray transform, and hence, one needs additional data to get the full recovery of the tensor fields. Integral moments transforms, and transverse ray transforms are such integral transforms that allow the full recovery of the unknown tensor fields. The interested reader may refer to \cite{UCP_2022,Denisjuk_2006, Lionheart_Naeem_2016,Griesmaier_TRT_2018, Holman_2009,Momentum_ray_transform_2018,Momentum_ray_transform_2_2020,Rohit_2021,Rohit_2024,Sharafutdinov_Novikov_2007,Sharafutdinov_Saint-Venant,Sharafutdinov_2007,Tuy_1983,Vertgeim_2000} for additional details on such problems.

\par Previously, the problem of recovering a 3-D vector field from joint measurements of electromagnetic and acoustic signals was considered in \cite{Kunyansky_2023}. This problem arises in an imaging paradigm called magnetoacoustoelectric tomography, wherein the data are given by the longitudinal Radon transform and the weighted Radon
transforms of the corresponding vector field. Recall that the Radon transform acting on vector fields has a non-trivial kernel, hence it is only possible to recover the vector field partially from the Radon transform data. In order to recover the vector field fully, the authors of \cite{Kunyansky_2023} show that it suffices to consider linearly weighted Radon transforms along with unweighted Radon transforms of such vector fields in $\mathbb{R}^n$. We would also like to note the important contributions of the authors of \cite{Polyakova_normal_Radon_transform}, who have proposed a complementary approach for the inversion of the Radon transform on vector fields using the method of singular value decomposition. Furthermore, the numerical results for the same setup were also provided in \cite{Polyakova_Svetov_normal_Radon_numerical}.
\par Recently, the authors in \cite{Polyakova_Svetov_Gen_Rad_tensor_fields} have studied the kernel of the generalized Radon transforms and provided reconstruction algorithms for vector fields in $\Rb^n$ and symmetric $m$-tensor fields in $\Rb^3$ from the data of longitudinal, transversal and mixed Radon transforms.
Our work can be considered as a generalization of \cite{Kunyansky_2023, Polyakova_Svetov_Gen_Rad_tensor_fields} to symmetric $m$-tensor fields in $\Rb^n$. We start by discussing the kernel of longitudinal and transversal Radon transforms, and it is worth noting that the kernels are as expected from our experience with straight-line transforms of tensor fields. Further, we show that the full tensor field $f$ can be determined from two different sets of data: (1) a combination of longitudinal Radon transform and corresponding weighted longitudinal Radon transform of the tensor field, (2) a combination of transversal Radon transform along with weighted transversal Radon transform of the tensor field.

The rest of the article is organized as follows: In Section \ref{Sec: Definitions}, we fix the notations, define the unweighted and weighted Radon transforms acting on symmetric $m$-tensor fields in $\mathbb{R}^n$, and state the main results of this article. Section \ref{Sec:kernel description} discusses the proof of the kernel description result. Sections \ref{Sec: Reconstruction using LRT and WLRT} and \ref{Sec: Reconstruction using TRT and WTRT} are devoted to proving the main theorems Theorem \ref{th:inversion using longitudinal} and \ref{th:inversion using transversal}.

\section{Definitions and Main results} \label{Sec: Definitions}
In this section, we rigorously define longitudinal and transversal Radon transforms of symmetric $m$-tensor fields in $\Rb^n$. We also make auxiliary definitions for the weighted longitudinal and weighted transversal Radon transforms. We present a decomposition result for symmetric $m$-tensor fields and introduce some operators acting on tensor fields. Further, we discuss some known facts about the classical Radon transform that we will need in the upcoming analysis. Finally, we end the section with statements of our main theorems and remarks about these.

Let $\textit{T}^m(\Rb^n)$ denote the space of $m$-tensor fields defined in $\Rb^n$ and $\mathit{S}^m(\Rb^n)$ be the space of symmetric $m$-tensor fields. Further, the Schwartz space  $\mathcal{S} (\mathit{S}^m(\Rb^n))$ consist of symmetric $m$-tensor fields whose components belong to the Schwartz space $\mathcal{S}(\Rb^n)$ of functions in $\Rb^n$. The Hilbert space $H^s_t(\Rb^n); s \in \Rb, t > -n/2$ is defined to be the completion of $\mathcal{S}(\Rb^n)$ with respect to the following norm (see \cite{Sharafutdinov_Reshetnyak_formula} for a detailed discussion about these spaces):
$$||f||^2_{H^s_t(\Rb^n)} = \int_{\Rb^n} |y|^{2 t} (1 + |y|^2)^{s - t} |\hat{f}(y)|^2 \,dy.$$
When $t = 0$, we get the classical Sobolev space $H^s(\Rb^n)$. We will be mostly working with tensor fields in the space $H^s_t (\mathit{S}^m(\Rb^n))$ whose components lie in the space $H^s_t(\Rb^n)$.

For $\o \in \Sb^{n-1}$, $\o^\perp$ denotes the hyperplane passing through the origin and perpendicular to $\o$. Let $\o_1, \o_2, \dots, \o_{n-1}$ be an orthonormal basis of $\o^\perp$. Thus, the set $\{\o, \o_1, \o_2, \dots, \o_{n-1}\}$ forms an orthonormal frame in $\Rb^n$. Further, each pair $(\o, p) \in \mathbb{S}^{n - 1} \times \Rb$ represents a unique hyperplane $H_{\o, p}$ in $\Rb^n$ that is perpendicular to $\o$ and at a distance $p$ from the origin.

Given real s and $t > -1/2$, the Hilbert space $H^s_t(\Sb^{n-1} \times \Rb)$ is defined as the completion of $\Sc(\Sb^{n-1} \times \Rb)$ with respect to the norm (see \cite{Sharafutdinov_Reshetnyak_formula})
$$||\varphi||^2_{H^s_t(\Sb^{n-1} \times \Rb)} = \frac{1}{2(2\pi)^{n-1}}\int_{\Sb^{n-1}}\int_{\Rb} |\sigma|^{2 t} (1 + |\sigma|^2)^{s - t} |\hat{\varphi}(\o, \sigma)|^2 \,d\sigma\, d\o $$
where $d\o$ is the volume form on $\Sb^{n-1}$ which is induced by the Euclidean metric of $\Rb^n$ and $\hat{\varphi}(\o, p)$ is the one dimensional Fourier transform of $\varphi$ with respect to the variable $p$.

Now, we are ready to define the integral transforms of our interest. We define all the transforms for the Schwartz space $\mathcal{S} (\mathit{S}^m(\Rb^n))$ of symmetric $m$-tensor fields. However, all these definitions can be extended to tensor fields in the space $H^s_t (\mathit{S}^m(\Rb^n))$ by standard density arguments. Note that the Einstein summation convention of summing over repeated indices is used in the following definitions and throughout the article.
\begin{definition}
The \textbf{transversal Radon transform (TRT)} of a symmetric $m$-tensor field $f \in\mathcal{S} (\mathit{S}^m(\Rb^n))$ is defined as:
\begin{align}\label{eq:TRT}
\mathcal{T}^m f (\omega, p) = \int_{H_{\omega,p}} \left< f (x), \omega^{\odot m} \right> \,dx =  \int_{\omega^\perp} \left< f (p\omega + y), \omega^{\odot m} \right> \,dy, 
\end{align}
where $\omega^{\odot m}$ denotes the $m^{th}$ symmetric tensor product of $\omega$ and $\left< f(x), \omega^{\odot m} \right> = f_{i_1\dots i_m}(x) \omega^{i_1} \cdots \omega^{i_m}$.
\end{definition}
\begin{definition}
The \textbf{longitudinal Radon transform (LRT)} of a symmetric $m$-tensor field $f \in \mathcal{S} (\mathit{S}^m(\Rb^n))$ is defined as follows:
\begin{align}
\Lc_{\ell_1 \dots \ell_{n - 1}}^m f (\omega, p) = \int_{H_{\omega,p}} \left<f(x), \o_1^{\odot \ell_1} \odot \dots \odot \o_{n - 1}^{\odot \ell_{n - 1}}\right> \,dx
= \int_{\omega^\perp} \left< f (p\omega + y), \o_1^{\odot \ell_1} \odot \dots \odot \o_{n - 1}^{\odot \ell_{n - 1}} \right> \,dy,
\end{align}
where $\ell_i \geq 0$; $i = 1, \dots, n - 1$ and $\ell_1 + \dots + \ell_{n - 1} = m.$
\end{definition}
\begin{definition}
The \textbf{weighted transversal Radon transform (weighted TRT)} of order $k$, for $1 \leq k \leq m$, of a symmetric $m$-tensor field $f \in \mathcal{S} (\mathit{S}^m(\Rb^n))$ is defined as follows:
\begin{equation}\label{eq:WTRT}
\Tc^{w,k}_{\ell_1 \dots \ell_{n - 1}} f (\omega, p) = \int_{H_{\omega,p}} \left<x, \omega_1\right>^{\ell_1} \dots \left<x, \omega_{n - 1}\right>^{\ell_{n - 1}} \left< f (x), \omega^{\odot m} \right> \,dx
\end{equation}
\noindent such that $\ell_i \geq 0;$ $i = 1 \dots n - 1$ and $\ell_1 + \dots + \ell_{n - 1} = k$.
\end{definition}
\noindent Note that, there are total $ \binom{k + n - 2}{k}$ weighted transversal Radon transforms of order $k$. The weighted TRT can be understood to be a generalization of the well-known moment ray transform by making the following simple observation. For a fixed $\omega=\omega_n$ and for any $ y\in \omega^{\perp}$ we can write, $y = \sum_{j=1}^{n-1}y_j\o_j$. Then it is easy to see that:
\begin{align*}
  \Tc^{w,k}_{\ell_1 \dots \ell_{n - 1}} f (\omega, p) = \int_{\Rb} \int_{\Rb}\dots \int_{\Rb} y_1^{\ell_1}y_2^{\ell_2} \dots y_{n-1}^{\ell_{n - 1}} \left< f (p\omega + y), \omega^{\odot m} \right> \,dy_1\,dy_2\dots \,dy_{n-1}.   
\end{align*}
\begin{definition}
The \textbf{weighted longitudinal Radon transform (weighted LRT)} of order $k$ of a symmetric $m$-tensor field $f \in \mathcal{S} (\mathit{S}^m(\Rb^n))$, where $1\leq k\leq m$ is defined as follows:
\begin{align}
\Lc^{w,k}_{\ell_1 \dots \ell_{n - 1}} f (\omega, p) = \int_{H_{\omega,p}} \left<x, \o_1\right>^{k_1} \dots \left<x, \o_{n - 1}\right>^{k_{n - 1}} \left<f(x), \o_1^{\odot \ell_1 + k_1} \odot \dots \odot \o_{n - 1}^{\odot \ell_{n - 1} + k_{n - 1}}\right> \,dx
\end{align}
such that $k_i \geq 0$, $\ell_i \geq 0$ for $i = 1, \dots, n - 1$. Also, $k_1 + \dots + k_{n - 1} = k$ and $\ell_1 + \dots + \ell_{n - 1} = m - k.$
\end{definition}
\noindent Similar to our description for weighted TRT as above, we can rewrite the weighted LRT transform as follows:
\begin{align*}
\Lc^{w,k}_{\ell_1 \dots \ell_{n - 1}} f (\omega, p) = \int_{\Rb} \int_{\Rb}\dots \int_{\Rb} y_1^{k_1}y_2^{k_2} \dots y_{n-1}^{k_{n - 1}} \left<f(p\o+ y), \o_1^{\odot \ell_1 + k_1} \odot \dots \odot \o_{n - 1}^{\odot \ell_{n - 1} + k_{n - 1}}\right> \,dy_1\,dy_2\dots \,dy_{n-1}.
\end{align*}
For any given $k$ there are in total $ \binom{k + n - 2}{k} \times \binom{m - k + n - 2}{m - k}$ weighted longitudinal Radon transforms of order $k$. When $n=2$, the above definition of the weighted LRT reduces to the definition for the $k$-th integral moment of the longitudinal ray transform of a symmetric $m$-tensor field.

Now, we briefly discuss about the classical Radon transform of scalar fields. The \textbf{classical Radon transform} $\Rc : \mathcal{S}(\Rb^n) \rightarrow \mathcal{S}(\mathbb{S}^{n - 1} \times \Rb)$ of a scalar function $f \in \mathcal{S}(\Rb^n)$ is defined as follows:
\[f^\wedge(\omega,p) = \Rc f(\o,p) = \int_{H_{\omega,p}}f(x)\,ds\]
where $ds$ is the standard volume element on hyperplane $H_{\omega,p}$. The Radon transform extends uniquely to the bijective isometry of Hilbert spaces, see \cite[Theorem 2.1]{Sharafutdinov_Reshetnyak_formula}
$$\mathcal{R}: H^s_t(\Rb^n) \rightarrow H^{s + (n - 1)/2}_{t + (n - 1)/2, e}(\mathbb{S}^{n - 1} \times \Rb)$$ 
which means we can recover a function $f\in H^s_t(\Rb^n)$ from its Radon transform. Here, $H^{s + (n - 1)/2}_{t + (n - 1)/2, e}$ is a subspace of $H^{s + (n - 1)/2}_{t + (n - 1)/2}$ consisting of functions $\varphi$ that satisfy $\varphi (-\o, -p) = \varphi (\o, p)$. Also, the following property of the Radon transform is handy for our analysis: 
\begin{equation} \label{eq:property of Radon transform}
    \left(a_i\frac{\partial}{\partial x^i}f(x)\right)^\wedge(\omega,p) = \left<\omega,a\right> \frac{\partial}{\partial p} f^\wedge(\omega,p).
\end{equation}
For a symmetric $m$-tensor field $f \in \mathcal{S}(\mathit{S}^m(\Rb^n))$, the componentwise Radon transform of $f$ is again a symmetric $m$-tensor field denoted by $\overline{\Rc}f$ and is defined as follows:
$$(\overline{\Rc}f)_{i_1 \dots i_m} = \Rc (f_{i_1 \dots i_m}).$$

In the remaining section, we introduce some operators acting on tensor fields and their crucial properties, please refer \cite{Sharafutdinov_Book, Sharafutdinov_Reshetnyak_formula} for more details:
\begin{itemize}
\item The natural projection $\sigma : \textit{T}^m(\Rb^n) \rightarrow \mathit{S}^m(\Rb^n)$, which is also known as the  \textit{symmmetrization operator}, is defined as follows:
\begin{align*}
    \sigma (1, \dots, m) u (x_1, \dots, x_m) = \frac{1}{m!} \sum_{\pi \in {\Pi_m}} u (x_{\pi(1)}, \dots, x_{\pi(m)}),
\end{align*}
where $\Pi_m$ denotes the permutations group on the set $\left\{1, \dots, m\right\}$.
\item The operator of \textit{symmetric multiplication} $i_x : \mathit{S}^{m}(\Rb^n) \rightarrow \mathit{S}^{m + 1}(\Rb^n)$ is defined as follows:
\begin{equation*}
    (i_x u)_{i_1 \dots i_{m + 1}} = \sigma (i_1, \dots, i_{m + 1}) x_{i_{m + 1}} u_{i_m \dots i_m}.
\end{equation*}
\item The operator of \textit{contraction} $j_x : \mathit{S}^{m}(\Rb^n) \rightarrow \mathit{S}^{m - 1}(\Rb^n)$ is defined as follows:
\begin{equation*}
    (j_x u)_{i_1 \dots i_{m - 1}} = x^{i_m} u_{i_1 \dots i_m}.
\end{equation*}
\item The operator of \textit{inner differentiation} or symmetrized covariant derivative $\D : \mathcal{S} (\mathit{S}^{m}(\Rb^n)) \rightarrow \mathcal{S} (\mathit{S}^{m + 1}(\Rb^n))$ is defined as follows:
\begin{equation*}
    \left(\D u\right) _{i_1 \dots i_{m + 1}} = \sigma (i_1, \dots, i_{m + 1}) \frac{\partial  u_{i_1 \dots i_m}}{\partial x^{i_{m + 1}}}.
\end{equation*}
\item The \textit{divergence operator} $\delta : \mathcal{S} (\mathit{S}^{m}(\Rb^n)) \rightarrow \mathcal{S} (\mathit{S}^{m - 1}(\Rb^n))$ is defined as follows:
\begin{equation*}
    \left(\delta u\right) _{i_1 \dots i_{m - 1}} = \frac{\partial}{\partial x^{i}} u_{i_1 \dots i_{m - 1} i}.
\end{equation*}
\end{itemize}
Next, we discuss some crucial properties of the differential operator $\displaystyle D_j := -i \frac{\partial}{\partial x^j}$ and a result about the solvability of an elliptic system of equations in $H^s_t$ space setting that will be used repeatedly throughout the article: 
\begin{lemma}\cite[Lemma 3.3]{Sharafutdinov_Reshetnyak_formula}
    Every partial derivative $\displaystyle D_j : \mathcal{S}(\Rb^n) \rightarrow \mathcal{S}(\Rb^n)$ uniquely extends to the bounded operator $$D_j : H^{s + 1}_{t + 1}(\Rb^n) \rightarrow H^s_t(\Rb^n) \quad (s \in \Rb, t > -n/2).$$
\end{lemma}
\begin{theorem}\label{th:solvability of elliptic system}
For $ k \geq 0$ and $h \in H^s_t(\mathit{S^\ell}(\Rb^n))$, there exists a unique tensor field $v \in H^{s + 2k}_{t + 2k}(\mathit{S^\ell}(\Rb^n))$  that satisfies $$\delta^k \D^k v = h.$$    
\end{theorem}
\begin{proof}
    The Fourier transform of a tensor field is defined componentwise. Let $\hat{v}$ and $\hat{h}$ denote the Fourier transform of $v$ and $h$ respectively. Then using the properties of the Fourier transform, we get
    \begin{align*}
        \delta^k \D^k v &= h\\
        \implies j_x^k i_x^k \hat{v} &= \hat{h} \quad \mbox{($i_x$ and $j_x$ are defined above).}
    \end{align*}
    If the matrix formed by the entries of $j_x^k i_x^k$ is invertible, the $\hat{v}$ can be recovered, which further gives $v$ by Fourier inversion. In order to accomplish this, it is sufficient to prove that $\delta^k d^k v = 0$ implies $v = 0$. This has been proved for $k = 1$ in \cite[Lemma 3.4]{Sharafutdinov_Reshetnyak_formula}. To prove this for arbitrary $k$, an inequality of the form 
    $$|v| \leq C |x|^{-2k} |j_x^ki_x^k v|; \quad C \mbox{ is some constant},$$ 
    which is similar to \cite[Equation 3.4]{Sharafutdinov_Reshetnyak_formula}, can be derived using Equation \eqref{eq: for solenoidal fields}. Using this inequality, the result can be obtained for arbitrary $k$ following the steps from \cite[Lemma 3.4]{Sharafutdinov_Reshetnyak_formula}.
\end{proof}

\noindent Next, we state a theorem about the decomposition of a symmetric $m$-tensor field $f \in H^s_t(\textit{S}^m(\Rb^n))$.
\begin{theorem}\label{Decomposition Result}
    For any tensor field $f \in H^s_t(S^m(\mathbb{R}^n)) (s \in \mathbb{R}, t > -n/2, m \geq 0)$, there exist uniquely defined $v_0, \dots , v_m$ with $v_i \in H^{s + i}_{t + i}(S^{m-i}(\mathbb{R}^n))$ for $i = 0,1, \dots, m$ such that
    \begin{align}\label{eq:decomposition of f}
        f = \sum_{i=0}^m \D^iv_i, \quad \mbox{ with }  v_{i} \text{ solenoidal for } 0\leq i \leq  m-1.
    \end{align}
    The estimate 
    \begin{equation*}
        ||v_i||_{H^{s + i}_{t + i}(\mathbb{R}^n)} \leq C ||f||_{H^s_t(\mathbb{R}^n)}; \quad i= 0, \dots, m
    \end{equation*} holds for a constant $C$ independent of $f$. The tensor field $v_0$ is known as \textbf{the solenoidal part of $f$}, and the remaining \textbf{$\displaystyle \left(\D\sum_{i=1}^m \D^{i-1}v_i\right)$ is the potential part of $f$.}
\end{theorem}
\begin{proof}
    The proof directly follows by iterating the result in \cite[Theorem 3.5]{Sharafutdinov_Reshetnyak_formula}.
\end{proof}
Note to recover a symmetric $m$-tensor field, it is sufficient to recover the components $v_i$, $0\leq i \leq m$, coming from the above decomposition. This is what we claim in Theorems \ref{th:inversion using longitudinal} and \ref{th:inversion using transversal} below, which are the main results of our work. But before that, we give kernel description for LRT and TRT in the following theorem.

\begin{theorem}\label{th:kernel description}[Kernel description]
For $s \in \mathbb{R}$ and $t > -\frac{n}{2}$, a symmetric $m$-tensor field $f \in H^s_t(S^m(\Rb^n))$ satisfy 
\begin{enumerate}
    \item[(a)] $\Lc_{\ell_1 \dots \ell_{n - 1}}^m f =0 $ if and only if $f = \D v$ for some $v \in H^{s+1}_{t+1}(S^m(\Rb^n)).$
    \item[(b)] $\Tc^m f = 0$ if and only if $f =  \sum_{i=0}^{m-1} \D^iv_i$, for  $v_i \in H^{s+i}_{t+i}(S^m(\Rb^n))$ and each $v_i$ is divergence-free. 
\end{enumerate}
\end{theorem}
\begin{theorem}\label{th:inversion using longitudinal}[Recovery of $f$ from LRT and weighted LRT data]
For $s \in \mathbb{R}$ and $t > -\frac{n}{2}$, a symmetric $m$-tensor field $f \in H^s_t(S^m(\Rb^n))$ can be reconstructed uniquely from a combination of its longitudinal and weighted longitudinal Radon transform. We show that the solenoidal part $v_0$ (as in the above decomposition Theorem \ref{Decomposition Result}) is uniquely determined from its longitudinal Radon transform $\Lc_{\ell_1 \dots \ell_{n - 1}}^m f$, and the remaining potential part $v_k$, $k = 1, \dots, m$, can be reconstructed from weighted longitudinal Radon transform of order $k$, $\Lc_{\ell_1 \dots \ell_{n - 1}}^{w,k} f$. 
\end{theorem}

\begin{theorem}\label{th:inversion using transversal}[Recovery of $f$ from TRT and weighted TRT data]
For $s \in \mathbb{R}$ and $t > -\frac{n}{2}$, a symmetric $m$-tensor field $f \in H^{s+m}_t(S^m(\Rb^n))$ can be reconstructed uniquely from a combination of its transversal and weighted transversal Radon transform. We show that the solenoidal part $v_m$ (as in the above decomposition Theorem \ref{Decomposition Result}) is uniquely determined from its transversal Radon transform $\Tc^m f$, and the remaining components $v_{m - k}$, $k = 1, \dots, m$, can be reconstructed from weighted transversal Radon transform of order $k$, $\Tc_{\ell_1 \dots \ell_{n - 1}}^{w,k} f$.
\end{theorem}
\begin{remark}
    Out of the total $\binom{k + n - 2}{k} \times \binom{m - k + n - 2}{m - k}$ weighted longitudinal Radon transforms of order $k$, any $\binom{k + n - 2}{k}$ can be considered to carry out the reconstruction similar to the reconstruction done in Section \ref{Sec: Reconstruction using LRT and WLRT}. For the simplicity of notation, we have considered the $\binom{k + n - 2}{k}$ transforms having $k_1 = k$ and $k_i = 0$ for $i \neq 1$, that is, we have shown the calculation for transforms of the form:
    \begin{align}
        \Lc^{w,k}_{\ell_1 \dots \ell_{n - 1}} f (\omega, p) = \int_{H_{\omega,p}} \left<x, \o_1\right>^{k} \left<f(x), \o_1^{\odot \ell_1 + k} \odot \dots \odot \o_{n - 1}^{\odot \ell_{n - 1}}\right> \,dx
    \end{align}
\end{remark}
\section{Proof of Theorem \ref{th:kernel description} (Kernel description for LRT and TRT)}\label{Sec:kernel description}
\begin{proof}[Proof of part (a)]
We prove the result here in one direction only; the reverse direction follows directly from a combination of the decomposition result Theorem \ref{Decomposition Result} and the inversion result Theorem \ref{th:inversion using longitudinal} (proved in the next section). Let $f = \D v$. Then we have
\begin{align*}
\Lc^m_{\ell_1\dots \ell_{n - 1}} f(\o, p) &= \int_{H_{\omega,p}} \left<(\D v)(x), \o_1^{\odot \ell_1} \odot \dots \odot \o_{n - 1}^{\odot \ell_{n - 1}}\right> \,dx \\
&= \sigma(i_1, \dots, i_m) \int_{H_{\omega,p}} \frac{\partial v_{i_1\dots i_{m-1}}}{\partial x_{i_m}} (x) \left(\o_1^{\odot \ell_1} \odot \dots \odot \o_{n - 1}^{\odot \ell_{n - 1}}\right)^{i_1 \dots i_m} \,dx\\
&= 0 \quad \mbox{ using property \eqref{eq:property of Radon transform} of Radon transform.}
\end{align*}
This completes the proof in the forward direction.
\end{proof}
\begin{proof}[Proof of part (b)]
Before proving that $\displaystyle  f = \sum_{i=0}^{m-1} \D^iv_i$ implies $\Tc^m f = 0$, we prove the following two results:
\begin{enumerate}[label=(\Alph*)]
\item If $v$ is a solenoidal $m$-tensor field with $m \geq 1$, then for any non-negative integers $\ell_i; i = 1, \dots, n$ such that $\ell_1 + \dots + \ell_n = m - 1$, we have
$$\int_{H_{\o, p}} \left<v(x), \o_1^{\odot \ell_1} \odot \dots \odot \o_{n - 1}^{\odot \ell_{n - 1}} \odot \o^{\odot \ell_n + 1}\right> = 0; \quad \text{for all } (\o, p) \in \mathbb{S}^{n - 1} \times \Rb.$$
\begin{proof}[Proof of (A)]
For a fixed $\o \in \Sb^{n-1}$, we define a vector field 
$$\tilde{f}(x) = \left<v(x), \o_1^{\odot \ell_1} \odot \dots \odot \o_{n - 1}^{\odot \ell_{n - 1}} \odot \o^{\odot \ell_n}\right>.$$
Then $$\delta \tilde{f} =  \left<\delta v, \o_1^{\odot \ell_1} \odot \dots \odot \o_{n - 1}^{\odot \ell_{n - 1}} \odot \o^{\odot \ell_n}\right> = 0$$
since $v$ is solenoidal. Thus $\tilde{v}$ is a divergence-free vector field in $\Rb^n$. Hence, using \cite[Proposition 4]{Kunyansky_2023}, we get 
$$\Tc ^{1}\tilde{v} (\o, p) = \int_{H_{\o, p}} \left<v(x), \o_1^{\odot \ell_1} \odot \dots \odot \o_{n - 1}^{\odot \ell_{n - 1}} \odot \o^{\odot \ell_n + 1}\right> = 0.$$
Note that when $\ell_n = m - 1$, this result states that the solenoidal $m$-tensor field $v$ lies in the kernel of the transversal Radon transform $\mathcal{T}^m f$.
\end{proof}
\item For a symmetric $(m-1)$-tensor field $v$, the following identity holds:
$$ \Tc ^m(\D v) (\o, p) = \frac{\partial }{\partial p} \Tc^{m-1}(v)(\o,p).$$
\begin{proof}[Proof of (B)]
Consider
\begin{align*}
 \Tc^m(\D v) (\o, p) &= \int_{H_{\omega,p}} \left<\left(\D v\right) (x), \o^{\odot m}\right> \,dx\\
&=\int_{H_{\omega,p}} \left(\D v\right)_{i_1\dots i_m}(x)\o^{i_1}\cdots \o^{i_m}\,dx\\
&= \int_{H_{\omega,p}} \frac{\partial v_{i_1\dots i_{m-1}}}{\partial x_{i_m}} (x)\o^{i_1}\cdots \o^{i_m}\,dx\\
&= \langle \o, \o \rangle \frac{\partial }{\partial p}\int_{H_{\omega,p}} v_{i_1\dots i_{m-1}} (x)\o^{i_1}\cdots \o^{i_{m-1}}\,dx, \quad \mbox{ using property \eqref{eq:property of Radon transform}}\\
&= \frac{\partial }{\partial p}\Tc ^{m-1}(v)(\o, p).
\end{align*}
\end{proof}
\end{enumerate}
In the above result, the left-hand side has the transversal Radon transform of $m$-ordered tensor field $(\D v)$ whereas the right-hand side contains the transversal Radon transform of $(m - 1)$-ordered tensor field $v$. Using the above proposition repeatedly, we get the following identity for a $(m-i)$-ordered symmetric tensor field $v$:
$$ \Tc^m \left(\D^i v\right)(\o, p)  = \frac{\partial^i }{\partial p^i}\Tc^{m-i} (v)(\o,p).  $$
Now for $f = \sum_{i=0}^{m-1} \D^iv_i$ with $\delta v_i = 0$ for all $i = 1, \dots, m - 1$, we have 
\begin{align*}
\sum_{i=0}^{m-1} \Tc ^m\left(\D^i v_i\right) (\o,p) &= \sum_{i=0}^{m-1} \int_{H_{\omega,p}} \left<\left(\D^i v_i\right) (x), \o^{\odot m}\right>\,dx\\ 
&= \sum_{i=0}^{m-1} \frac{\partial^i }{\partial p^i} \Tc^{m-i} v_i (\o,p), \quad \mbox{ using result (B)}\\
&= 0, \quad \mbox{ using result (A)}.
\end{align*}
The converse can be proved by showing that $v_m$ (from the decomposition given in equation \eqref{eq:decomposition of f}) can be recovered from the data of the transversal Radon transform $\Tc^m f$. This is done in Subsection \ref{subsec: inversion of TRT}.
\end{proof}

\section{Proof of Theorem \ref{th:inversion using longitudinal}} \label{Sec: Reconstruction using LRT and WLRT}
We divide the proof of Theorem \ref{th:inversion using longitudinal} in three subsections; the first subsection will show the solenoidal part $v_0$ of $f$ can be recovered from $\Lc_{\ell_1 \dots \ell_{n - 1}}^m f$ whereas the other two subsections will address the recovery of the potential part from the weighted longitudinal transforms.
\subsection{Recovery of \texorpdfstring{$v_0$}{v0} from \texorpdfstring{$\Lc^m_{\ell_1 \dots \ell_{n - 1}} f$}{Lcm ell1... elln-1 f}} \label{subsec: inversion of LRT}
\begin{proof}[Proof of recovering of $v_0$ using LRT]
It is sufficient to show that the componentwise Radon transform $\overline{\Rc}$ of $v_0$ can be recovered from the knowledge of longitudinal Radon transforms. Since $\o, \o_1, \o_2, \dots, \o_{n - 1}$ are linearly independent vectors in $\FR^n$, the collection of $\binom{m + n - 1}{k}$ $m$-tensors of the form 
$$\o_1^{\odot \ell_1} \odot \o_2^{\odot \ell_2} \odot \dots \odot \o_{n - 1}^{\odot \ell_{n - 1}} \odot \o^{\odot \ell_n}$$
\noindent such that $\ell_1 + \dots + \ell_n = m$ are linearly independent using \cite[Lemma 6]{Microlocal_2018}. It is known that the dimension of the space of $m$ tensors in $\FR^n$ is also $\binom{m + n - 1}{m}$, hence $v_0$ at any point $x \in \Rb^n$ can be written as follows:
\begin{align*}
    v_0 (x) &= \sum_{\ell_1 + \dots + \ell_n = m} \left<v_0 (x), \o_1^{\odot \ell_1} \odot \o_2^{\odot \ell_2} \odot \dots \odot \o_{n - 1}^{\odot \ell_{n - 1}} \odot \o^{\odot \ell_n} \right> \o_1^{\odot \ell_1} \odot \o_2^{\odot \ell_2} \odot \dots \odot \o_{n - 1}^{\odot \ell_{n - 1}} \odot \o^{\odot \ell_n}\\ 
    &= \sum_{\ell_1 + \dots + \ell_n = m - 1} \left<v_0 (x), \o \odot \o_1^{\odot \ell_1} \odot \o_2^{\odot \ell_2} \odot \dots \odot \o_{n - 1}^{\odot \ell_{n - 1}} \odot \o^{\odot \ell_n} \right> \o_1^{\odot \ell_1} \odot \o_2^{\odot \ell_2} \odot \dots \odot \o_{n - 1}^{\odot \ell_{n - 1}} \odot \o^{\odot \ell_n + 1}\\
    &+ \sum_{\ell_1 + \dots + \ell_{n - 1} = m} \left< v_0 (x), \o_1^{\odot \ell_1} \odot \o_2^{\odot \ell_2} \odot \dots \odot \o_{n - 1}^{\odot \ell_{n - 1}} \right> \o_1^{\odot \ell_1} \odot \o_2^{\odot \ell_2} \odot \dots \odot \o_{n - 1}^{\odot \ell_{n - 1}}. 
\end{align*}
In the above equation, the summation over $\ell_1 + \dots + \ell_n = m$ has been split into two cases: one when $\ell_n = 0$ and the other when $\ell_n \geq 1$. Hence, the componentwise Radon transform of $v_0$ can be written as follows:
\begin{align}\label{eq: componentwise Radon transform}
\begin{split}
    \overline{\Rc} v_0 &= \sum_{\ell_1 + \dots + \ell_n = m} \left<\overline{\Rc} v_0, \o_1^{\odot \ell_1} \odot \dots \odot \o_{n - 1}^{\odot \ell_{n - 1}} \odot \o^{\odot \ell_n}\right> \o_1^{\odot \ell_1} \odot \o_2^{\odot \ell_2} \odot \dots \odot \o_{n - 1}^{\odot \ell_{n - 1}} \odot \o^{\odot \ell_n}\\
    &= \sum_{\ell_1 + \dots + \ell_n = m - 1} \left<\overline{\Rc} v_0, \o \odot \o_1^{\odot \ell_1} \odot \dots \odot \o_{n - 1}^{\odot \ell_{n - 1}} \odot \o^{\odot \ell_n}\right> \o_1^{\odot \ell_1} \odot \o_2^{\odot \ell_2} \odot \dots \odot \o_{n - 1}^{\odot \ell_{n - 1}} \odot \o^{\odot \ell_n + 1}\\
    &+ \sum_{\ell_1 + \dots + \ell_{n - 1} = m} \left<\overline{\Rc} v_0, \o_1^{\odot \ell_1} \odot \dots \odot \o_{n - 1}^{\odot \ell_{n - 1}}\right> \o_1^{\odot \ell_1} \odot \o_2^{\odot \ell_2} \odot \dots \odot \o_{n - 1}^{\odot \ell_{n - 1}}
\end{split}
\end{align}
\noindent The first summation in the above equation is zero from claim A of part (b) of Theorem \ref{th:kernel description}. Also note that 
$$\left<\overline{\Rc} v_0, \o_1^{\odot \ell_1} \odot \dots \odot \o_{n - 1}^{\odot \ell_{n - 1}}\right> = \Lc_{\ell_1 \dots \ell_{n - 1}}^m f.$$
\noindent So the right-hand side in the above equation is completely known from the given data; hence, $\overline{\Rc} v_0$ can be recovered and therefore $v_0$ can be recovered using invertibility of the Radon transform. This completes the proof of our claim.
\end{proof}
Next, we prove that $v_k$ for $k = 1, \dots, m$ can be recovered from weighted longitudinal Radon transforms of order $k$. This is proved using mathematical induction.
\subsection{Recovery of \texorpdfstring{$v_1$}{v1} from \texorpdfstring{$\Lc^{w,1}_{\ell_1 \dots \ell_{n - 1}} f$}{Lcm ell1... elln-1 f}}\label{Subsec: linearly weighted LRT}
\begin{proof}[Proof of recovering of $v_1$ from $\Lc^{w,1}_{\ell_1 \dots \ell_{n - 1}} f$]
Let $\overline{\Delta} v_1$ denote the componentwise Laplacian of the tensor field $v_1$, that is, $$(\overline{\Delta} v_1)_{i_1 \dots i_{m - 1}} (x) = \left(\frac{\partial^2}{\partial x_1^2} + \dots + \frac{\partial^2}{\partial x_n^2}\right) v_{i_1 \dots i_{m - 1}}(x).$$ The idea here is to obtain, the componentwise Radon transform $\overline{\Rc}$ of the componentwise Laplacian of $v_1$, using the following formula:
\begin{align} \label{eq: componentwise Radon transform WLRT1}
\overline{\Rc} \left(\overline{\Delta} v_1\right) &= \sum_{\ell_1 + \dots + \ell_n = m - 1} \left<\overline{\Rc} \left(\overline{\Delta} v_1\right), \o_1^{\odot \ell_1} \odot \dots \odot \o_{n - 1}^{\odot \ell_{n - 1}} \odot \o^{\odot \ell_n}\right> \o_1^{\odot \ell_1} \odot \dots \odot \o_{n - 1}^{\odot \ell_{n - 1}} \odot \o^{\odot \ell_n}\nonumber \\
&= \sum_{\ell_1 + \dots + \ell_n = m - 2} \left<\overline{\Rc} \left(\overline{\Delta} v_1\right), \o \odot \o_1^{\odot \ell_1} \odot \dots \odot \o_{n - 1}^{\odot \ell_{n - 1}} \odot \o^{\odot \ell_n}\right> \o \odot \o_1^{\odot \ell_1} \odot \dots \odot \o_{n - 1}^{\odot \ell_{n - 1}} \odot \o^{\odot \ell_n}\nonumber \\
&\qquad + \sum_{\ell_1 + \dots + \ell_{n - 1} = m - 1} \left<\overline{\Rc} \left(\overline{\Delta} v_1\right), \o_1^{\odot \ell_1} \odot \dots \odot \o_{n - 1}^{\odot \ell_{n - 1}}\right> \o_1^{\odot \ell_1} \odot \dots \odot \o_{n - 1}^{\odot \ell_{n - 1}}
\end{align}
The tensor field $\overline{\Delta} v_1$ is solenoidal since $v_1$ is solenoidal, hence the first summation on the right-hand side of the above equation is zero using 2(A) of Theorem \ref{th:kernel description}. Next we show that 
$$\left<\overline{\Rc} \left(\overline{\Delta} v_1\right), \o_1^{\odot \ell_1} \odot \dots \odot \o_{n - 1}^{\odot \ell_{n - 1}}\right>$$
can be written in terms of $v_0$ and $\Lc^{w,1}_{\ell_1 \dots \ell_{n - 1}} f$. Let $e_i$ for $i = 1, \dots, n$ denote the standard basis of $\FR^n$ and $\left<\cdot, \cdot\right>$ denote the usual dot product of vectors in $\FR^n$. Then for any $1 \leq i \leq n$,
\begin{align} \label{eq: derivative of weighted LRT order 1}
    \left<e_i, \o\right> \frac{\partial}{\partial p} \Lc^{w,1}_{\ell_1 \dots \ell_{n - 1}} f &= \left<e_i, \o\right> \frac{\partial}{\partial p} \Rc \left(\left<x, \o_1\right> \left<f(x), \o_1^{\odot \ell_1 + 1} \odot \o_2^{\odot \ell_2} \odot \dots \odot \o_{n - 1}^{\odot \ell_{n - 1}}\right>\right)\nonumber \\
    &= \Rc \left(\frac{\partial}{\partial x_i} \left\{\left<x, \o_1\right> \left<f(x), \o_1^{\odot \ell_1 + 1} \odot \o_2^{\odot \ell_2} \odot \dots \odot \o_{n - 1}^{\odot \ell_{n - 1}}\right>\right\}\right)\nonumber \\
    &= \Rc \left(\frac{\partial}{\partial x_i} \left\{\left<x, \o_1\right> \left<v_0 (x), \o_1^{\odot \ell_1 + 1} \odot \o_2^{\odot \ell_2} \odot \dots \odot \o_{n - 1}^{\odot \ell_{n - 1}}\right>\right\}\right)\nonumber \\
    &+ \underbrace{\Rc \left(\frac{\partial}{\partial x_i} \left\{\left<x, \o_1\right> \left<\D(v_1 + \D v_2 + \dots + \D^{m - 1} v_m) (x), \o_1^{\odot \ell_1 + 1} \odot \o_2^{\odot \ell_2} \odot \dots \odot \o_{n - 1}^{\odot \ell_{n - 1}}\right>\right\}\right)}_{I_1}.
\end{align}
Furthermore,
\begin{align*}
    I_1 &= \Rc \left(\left<e_i, \o_1\right> \left<\D(v_1 + \D v_2 + \dots + \D^{m - 1} v_m) (x), \o_1^{\odot \ell_1 + 1} \odot \o_2^{\odot \ell_2} \odot \dots \odot \o_{n - 1}^{\odot \ell_{n - 1}}\right>\right)\\
    &\qquad + \Rc \left(\left<x, \o_1\right> \frac{\partial}{\partial x_i} \left<\D(v_1 + \D v_2 + \dots + \D^{m - 1} v_m) (x), \o_1^{\odot \ell_1 + 1} \odot \o_2^{\odot \ell_2} \odot \dots \odot \o_{n - 1}^{\odot \ell_{n - 1}}\right>\right).
\end{align*}
The first part on the right-hand side of the above equation is zero because it is the scalar multiple of the longitudinal Radon transform of a potential tensor field. Let $\nabla = \left(\frac{\partial}{\partial x_1}, \dots, \frac{\partial}{\partial x_n}\right)$, then for any symmetric $(k - 1)$-ordered tensor field $f$ and any vector $\o$, we have
\begin{align*}
    \left<\D f, \o^{\odot k}\right> &= \o^{i_1} \dots \o^{i_k} \frac{\partial}{\partial x_{i_k}} f_{i_1 \dots i_{k - 1}} =   \left<\o, \nabla\right> \left<f, \o^{\odot k - 1}\right>.
\end{align*}
\noindent Using this relation, $I_1$ can be rewritten as follows:
\begin{align*}
    I_1 &= \Rc \left(\left<x, \o_1\right> \frac{\partial}{\partial x_i} \left<\D(v_1 + \D v_2 + \dots + \D^{m - 1} v_m) (x), \o_1^{\odot \ell_1 + 1} \odot \o_2^{\odot \ell_2} \odot \dots \odot \o_{n - 1}^{\odot \ell_{n - 1}}\right>\right)\\
    &= \Rc \left(\left<x, \o_1\right> \left<\o_1, \nabla\right> \left<\frac{\partial}{\partial x_i} (v_1 + \D v_2 + \dots + \D^{m - 1} v_m) (x), \o_1^{\odot \ell_1} \odot \o_2^{\odot \ell_2} \odot \dots \odot \o_{n - 1}^{\odot \ell_{n - 1}}\right>\right).
\end{align*}
Any point $x$ in the hyperplane can be represented as $x = p \o + y_1 \o_1 + \dots + y_{n - 1} \o_{n - 1}$ for some $y_1, \dots, y_{n - 1} \in \FR$. Using this representation, the above equation can be written as follows:\\
\begin{align*}
    I_1 &= \int_\FR \dots \int_\FR y_1 \frac{\partial}{\partial y_1} \left<\frac{\partial}{\partial x_i} (v_1 + \D v_2 + \dots + \D^{m - 1} v_m) (x), \o_1^{\odot \ell_1} \odot \o_2^{\odot \ell_2} \odot \dots \odot \o_{n - 1}^{\odot \ell_{n - 1}}\right> \,dy_1 \dots \,dy_{n - 1}\\
    &= - \int_\FR \dots \int_\FR \left<\frac{\partial}{\partial x_i} (v_1 + \D v_2 + \dots + \D^{m - 1} v_m) (x), \o_1^{\odot \ell_1} \odot \o_2^{\odot \ell_2} \odot \dots \odot \o_{n - 1}^{\odot \ell_{n - 1}}\right> \,dy_1 \dots \,dy_{n - 1}\\
    &= - \int_\FR \dots \int_\FR \left<\frac{\partial}{\partial x_i} (v_1) (x), \o_1^{\odot \ell_1} \odot \o_2^{\odot \ell_2} \odot \dots \odot \o_{n - 1}^{\odot \ell_{n - 1}}\right> \,dy_1 \dots \,dy_{n - 1}\\
    & \hspace{0.5cm} - \underbrace{\int_\FR \dots \int_\FR \left<\frac{\partial}{\partial x_i} (\D v_2 + \dots + \D^{m - 1} v_m) (x), \o_1^{\odot \ell_1} \odot \o_2^{\odot \ell_2} \odot \dots \odot \o_{n - 1}^{\odot \ell_{n - 1}}\right> \,dy_1 \dots \,dy_{n - 1}.}_{I_2}
\end{align*}
Now there exists $1 \leq j \leq n - 1$ such that $\ell_j \neq 0$. Without loss of generality, let $\ell_1 \neq 0$, then we get:
\begin{align*}
    &\left<\frac{\partial}{\partial x_i} (\D v_2 + \dots + \D^{m - 1} v_m) (x), \o_1^{\odot \ell_1} \odot \o_2^{\odot \ell_2} \odot \dots \odot \o_{n - 1}^{\odot \ell_{n - 1}}\right>\\ & \hspace{3cm}= \left<\o_1, \nabla\right> \left<\frac{\partial}{\partial x_i} (v_2 + \dots + \D^{m - 2} v_m) (x), \o_1^{\odot \ell_1 - 1} \odot \o_2^{\odot \ell_2} \odot \dots \odot \o_{n - 1}^{\odot \ell_{n - 1}}\right>.
\end{align*}
So the integral $I_2$ becomes
\begin{align*}
    I_2 &= \int_\FR \dots \int_\FR \frac{\partial}{\partial y_1} \left<\frac{\partial}{\partial x_i} (v_2 + \dots + \D^{m - 2} v_m) (x), \o_1^{\odot \ell_1 - 1} \odot \o_2^{\odot \ell_2} \odot \dots \odot \o_{n - 1}^{\odot \ell_{n - 1}}\right> \,dy_1 \dots \,dy_{n - 1}\\
    &= 0,
\end{align*}
using integration by parts. Finally, substituting the value of $I_1$ in equation \eqref{eq: derivative of weighted LRT order 1}, we get:
\begin{align*}
    &\Rc \left(\left<\frac{\partial}{\partial x_i} (v_1) (x), \o_1^{\odot \ell_1} \odot \o_2^{\odot \ell_2} \odot \dots \odot \o_{n - 1}^{\odot \ell_{n - 1}}\right>\right)\\
    &\hspace{1cm}= \Rc \left(\frac{\partial}{\partial x_i} \left<x, \o_1\right> \left<v_0 (x), \o_1^{\odot \ell_1 + 1} \odot \o_2^{\odot \ell_2} \odot \dots \odot \o_{n - 1}^{\odot \ell_{n - 1}}\right>\right) - \left<e_i, \o\right> \frac{\partial}{\partial p} \Lc^{w,1}_{\ell_1 \dots \ell_{n - 1}} f.
\end{align*}
Applying the operator $\left<e_i, \o\right> \frac{\partial}{\partial p}$ to the above equation, we get
\begin{align*}
    &\Rc \left(\left<\frac{\partial^2}{\partial {x_i}^2} (v_1) (x), \o_1^{\odot \ell_1} \odot \o_2^{\odot \ell_2} \odot \dots \odot \o_{n - 1}^{\odot \ell_{n - 1}}\right>\right)\\
    &\hspace{1cm}= \Rc \left(\frac{\partial^2}{\partial {x_i}^2} \left<x, \o_1\right> \left<v_0 (x), \o_1^{\odot \ell_1 + 1} \odot \o_2^{\odot \ell_2} \odot \dots \odot \o_{n - 1}^{\odot \ell_{n - 1}}\right>\right) - {\left<e_i, \o\right>}^2 \frac{\partial^2}{\partial p^2} \Lc^{w,1}_{\ell_1 \dots \ell_{n - 1}} f.
\end{align*}
Adding the above equation for $i = 1, \dots, n$, we have
\begin{align*}
&\Rc \left(\Delta \left<v_1 (x), \o_1^{\odot \ell_1} \odot \o_2^{\odot \ell_2} \odot \dots \odot \o_{n - 1}^{\odot \ell_{n - 1}}\right>\right)\\    
&\quad \qquad =\left<\overline{\Rc}\left(\overline{\Delta}v_1 (x)\right), \o_1^{\odot \ell_1} \odot \o_2^{\odot \ell_2} \odot \dots \odot \o_{n - 1}^{\odot \ell_{n - 1}}\right>\\
&\quad \qquad = \Rc \left(\Delta \left<x, \o_1\right> \left<v_0 (x), \o_1^{\odot \ell_1 + 1} \odot \o_2^{\odot \ell_2} \odot \dots \odot \o_{n - 1}^{\odot \ell_{n - 1}}\right>\right) - \frac{\partial^2}{\partial p^2} \Lc^{w,1}_{\ell_1 \dots \ell_{n - 1}} f.
\end{align*}
Since $v_0 (x)$ and the weighted longitudinal Radon transforms of order 1 are known for all $\ell_1, \dots, \ell_{n - 1}$ (such that $\ell_1 + \dots + \ell_{n - 1} = m - 1$), the right-hand side of equation \eqref{eq: componentwise Radon transform WLRT1} is known and thus, the componentwise Radon transform of  $\overline{\Delta} v_1$ is known. This is the same as saying that we know $\overline{\Delta} v_1$, and hence we know the $v_1$ as required.
\end{proof}

Now assume that $v_0, v_1, \dots v_{k - 1}$ are known for any $2 \leq k \leq m$, then $v_k$ can be obtained from the transforms $\Lc^{w,k}_{\ell_1 \dots \ell_{n - 1}} f$ as shown in the following subsection: 
\subsection{Recovery of \texorpdfstring{$v_k$}{vk} from \texorpdfstring{$\Lc^{w,k}_{\ell_1 \dots \ell_{n - 1}} f$}{Lcm ell1... elln-1 f}} \label{Subsec: WLRT}

\begin{proof}[Proof of recovering of $v_k$ from $\Lc^{w,k}_{\ell_1 \dots \ell_{n - 1}} f$]
Following the steps similar to recovery of $v_1$ in the previous subsection, the componentwise Radon transform $\overline{\Rc}$ of the componentwise Laplacian of $v_k$, $\overline{\Delta} v_k$, can be obtained using the formula:
\begin{align} \label{eq: componentwise Radon transform WLRTk}
    \overline{\Rc} \left(\overline{\Delta} v_k\right) &= \sum_{\ell_1 + \dots + \ell_n = m - k} \left<\overline{\Rc} \left(\overline{\Delta} v_k\right), \o_1^{\odot \ell_1} \odot \dots \odot \o_{n - 1}^{\odot \ell_{n - 1}} \odot \o^{\odot \ell_n}\right> \o_1^{\odot \ell_1} \odot \dots \odot \o_{n - 1}^{\odot \ell_{n - 1}} \odot \o^{\odot \ell_n}\nonumber \\
    &= \sum_{\ell_1 + \dots + \ell_n = m - k - 1} \left<\overline{\Rc} \left(\overline{\Delta} v_k\right), \o \odot \o_1^{\odot \ell_1} \odot \dots \odot \o_{n - 1}^{\odot \ell_{n - 1}} \odot \o^{\odot \ell_n}\right> \o \odot \o_1^{\odot \ell_1} \odot \dots \odot \o_{n - 1}^{\odot \ell_{n - 1}} \odot \o^{\odot \ell_n}\nonumber \\
    &\qquad + \sum_{\ell_1 + \dots + \ell_{n - 1} = m - k} \left<\overline{\Rc} \left(\overline{\Delta} v_k\right), \o_1^{\odot \ell_1} \odot \dots \odot \o_{n - 1}^{\odot \ell_{n - 1}}\right> \o_1^{\odot \ell_1} \odot \dots \odot \o_{n - 1}^{\odot \ell_{n - 1}}
\end{align} 
Since $v_k$ is solenoidal for $k = 2, \dots, m - 1$, the first summation on the right-hand side of the above equation is zero using Theorem \ref{th:kernel description}. For $k = m$, the above decomposition is not required, and the following arguments hold true. We show that 
$$\left<\overline{\Rc} \left(\overline{\Delta} v_k\right), \o_1^{\odot \ell_1} \odot \dots \odot \o_{n - 1}^{\odot \ell_{n - 1}}\right>; \quad \ell_1 + \dots + \ell_{n - 1} = m - k$$
can be written in terms of $v_0, v_1, \dots, v_{k - 1}$ and $\Lc^{w,j}_{\ell_1 \dots \ell_{n - 1}} f$ transforms for $1 \leq j \leq k$.  For any $1 \leq i \leq n$, we get 

\begin{align} \label{eq: derivative of weighted LRT order k}
    &\left<e_i, \o\right> \frac{\partial}{\partial p} \Lc^{w,k}_{\ell_1 \dots \ell_{n - 1}} f\nonumber \\
    &\quad \qquad = \left<e_i, \o\right> \frac{\partial}{\partial p} \Rc \left(\left<x, \o_1\right>^k \left<f(x), \o_1^{\odot \ell_1 + k} \odot \o_2^{\odot \ell_2} \odot \dots \odot \o_{n - 1}^{\odot \ell_{n - 1}}\right>\right)\nonumber \\
    & \quad \qquad= \Rc \left(\frac{\partial}{\partial x_i} \left\{\left<x, \o_1\right>^k \left<f(x), \o_1^{\odot \ell_1 + k} \odot \o_2^{\odot \ell_2} \odot \dots \odot \o_{n - 1}^{\odot \ell_{n - 1}}\right>\right\}\right)\nonumber \\
    & \quad \qquad = \Rc \left(k \left<e_i, \o_1\right> \left<x, \o_1\right>^{k - 1} \left<f(x), \o_1^{\odot \ell_1 + k} \odot \o_2^{\odot \ell_2} \odot \dots \odot \o_{n - 1}^{\odot \ell_{n - 1}}\right>\right)\nonumber \\
    & \quad \qquad \quad + \Rc \left(\left<x, \o_1\right>^k \frac{\partial}{\partial x_i} \left<f(x), \o_1^{\odot \ell_1 + k} \odot \o_2^{\odot \ell_2} \odot \dots \odot \o_{n - 1}^{\odot \ell_{n - 1}}\right>\right)\nonumber \\
    & \quad \qquad = k \left<e_i, \o_1\right> \Lc^{w,k - 1}_{\ell_1+1 \ell_2 \dots \ell_{n - 1}} f\nonumber \\
    &\quad \qquad \quad + \Rc \left(\left<x, \o_1\right>^k \frac{\partial}{\partial x_i} \left<(v_0 + \D v_1 + \dots + \D^{k - 1} v_{k - 1}) (x), \o_1^{\odot \ell_1 + k} \odot \o_2^{\odot \ell_2} \odot \dots \odot \o_{n - 1}^{\odot \ell_{n - 1}}\right>\right)\nonumber \\
    & \quad \qquad \quad + \underbrace{\Rc \left(\left<x, \o_1\right>^k \frac{\partial}{\partial x_i} \left<\D^k(v_k + \D v_{k + 1} + \dots + \D^{m - k} v_m) (x), \o_1^{\odot \ell_1 + k} \odot \o_2^{\odot \ell_2} \odot \dots \odot \o_{n - 1}^{\odot \ell_{n - 1}}\right>\right)}_{I_3}.
\end{align}
Now, following the calculation for $I_1$ from the previous subsection and using integration by parts multiple times, we have
\begin{align*}
    I_3 &= \int_\FR \dots \int_\FR y_1^k \frac{\partial^k}{\partial {y_1}^k} \left<\frac{\partial}{\partial x_i} (v_k + \D v_{k + 1} + \dots + \D^{m - k} v_m) (x), \o_1^{\odot \ell_1} \odot \o_2^{\odot \ell_2} \odot \dots \odot \o_{n - 1}^{\odot \ell_{n - 1}}\right> \,dy_1 \dots \,dy_{n - 1}\\
    &= {(- 1)}^k k! \int_\FR \dots \int_\FR \left<\frac{\partial}{\partial x_i} (v_k + \D v_{k + 1} + \dots + \D^{m - k} v_m) (x), \o_1^{\odot \ell_1} \odot \o_2^{\odot \ell_2} \odot \dots \odot \o_{n - 1}^{\odot \ell_{n - 1}}\right> \,dy_1 \dots \,dy_{n - 1}\\
    &= {(- 1)}^k k! \int_\FR \dots \int_\FR \left<\frac{\partial}{\partial x_i} (v_k) (x), \o_1^{\odot \ell_1} \odot \o_2^{\odot \ell_2} \odot \dots \odot \o_{n - 1}^{\odot \ell_{n - 1}}\right> \,dy_1 \dots \,dy_{n - 1}\\
    & \quad + {(- 1)}^k k! \underbrace{\int_\FR \dots \int_\FR \left<\frac{\partial}{\partial x_i} (\D v_{k + 1} + \dots + \D^{m - k} v_m) (x), \o_1^{\odot \ell_1} \odot \o_2^{\odot \ell_2} \odot \dots \odot \o_{n - 1}^{\odot \ell_{n - 1}}\right> \,dy_1 \dots \,dy_{n - 1}.}_{I_4}
\end{align*}
Following a similar calculation as we did for simplifying $I_2$ in the previous subsection, we get $I_4 = 0$. Thus the equation \eqref{eq: derivative of weighted LRT order k} becomes:
\begin{align*}
&\Rc \left(\left<\frac{\partial}{\partial x_i} (v_k) (x), \o_1^{\odot \ell_1} \odot \o_2^{\odot \ell_2} \odot \dots \odot \o_{n - 1}^{\odot \ell_{n - 1}}\right>\right)\\
&\qquad \ \  = \frac{1}{{(- 1)}^k k!} \left[\left<e_i, \o\right> \frac{\partial}{\partial p} \Lc^{w,k}_{\ell_1 \dots \ell_{n - 1}} f - k \left<e_i, \o_1\right> \Lc^{w,k - 1}_{\ell_1+1 \ell_2 \dots \ell_{n - 1}} f\right.\\
&\qquad\qquad \ \ \left.- \Rc \left(\left<x, \o_1\right>^k \frac{\partial}{\partial x_i} \left<(v_0 + \D v_1 + \dots + \D^{k - 1} v_{k - 1}) (x), \o_1^{\odot \ell_1 + k} \odot \o_2^{\odot \ell_2} \odot \dots \odot \o_{n - 1}^{\odot \ell_{n - 1}}\right>\right)\right].
\end{align*}
Applying $\left<e_i, \o\right>$ to the above equation and adding for $i = 1, \dots, n$, we get 
$$\Rc \left(\Delta\left<(v_k) (x), \o_1^{\odot \ell_1} \odot \o_2^{\odot \ell_2} \odot \dots \odot \o_{n - 1}^{\odot \ell_{n - 1}}\right>\right),$$
that is, we know 
$$\left<\overline{\Rc} \left(\overline{\Delta}(v_k) (x)\right), \o_1^{\odot \ell_1} \odot \o_2^{\odot \ell_2} \odot \dots \odot \o_{n - 1}^{\odot \ell_{n - 1}}\right>$$
in terms of the known data. This is true for all $\ell_1, \dots, \ell_{n - 1}$ such that $\ell_1 + \dots + \ell_{n - 1} = m - k$, and therefore applying Radon inversion to equation \eqref{eq: componentwise Radon transform WLRTk}, we obtain componentwise Laplacian of $v_k$. This implies we can determine $v_k$ by solving the Laplacian componentwise.
\end{proof}
\section{Proof of Theorem \ref{th:inversion using transversal}} \label{Sec: Reconstruction using TRT and WTRT}
This section is again divided into three subsections; the first subsection deals with recovering $v_m$ of $f$ (appearing in the decomposition \eqref{eq:decomposition of f}) from $\Tc^m f$ while the other two subsections show the recovery of the other components $v_i$  for  $i = 0, \dots, m - 1$ from the weighted transversal Radon transforms. 
\subsection{Recovery of \texorpdfstring{$v_m$}{vm} from \texorpdfstring{$\Tc^m f$}{Tcmf}} \label{subsec: inversion of TRT}

We start by proving a lemma that we need to achieve the goal of this subsection. 
\begin{lemma}\label{lem: Radon of divergence}
    Let $g$ be a symmetric $m$-tensor field, then the following identity holds:
    \begin{align}\label{eq:Radon of divergence}
   \Rc (\delta^m g) (\o, p) = \frac{\partial^m}{\partial p^m}\Tc^m \left(g\right) (\o, p). 
    \end{align}
\end{lemma}
\begin{proof}
    This result is already proved for the vector field case in \cite[Lemma 6]{Kunyansky_2023}. The idea for this general case is the same. We sketch the proof here for the sake of completeness. Consider
\begin{align*}
\Rc \left(\delta^m g\right) (\o, p) &= \Rc \left(\frac{\partial}{\partial x_{i_m}}{(\delta^{m - 1} g)}_{i_1\dots i_m}\right) (\o, p)\\
&= \langle \o, e_{i_m}\rangle \frac{\partial}{\partial p} \Rc\left({(\delta^{m - 1} g)}_{i_1\dots i_m}\right)(\o, p) \quad \text{ (using equation \eqref{eq:property of Radon transform})}\\
&=\langle \o, e_{i_1}\rangle \cdots \langle \o, e_{i_m}\rangle \frac{\partial^m}{\partial p^m}  \Rc \left( g_{i_1\dots i_m}\right) (\o, p) \quad \text{(repeating the above step $(m - 1)$ more times)}\\
&= \frac{\partial^m}{\partial p^m} \Rc \left( g_{i_1\dots i_m}\o^{i_1}\cdots \o^{i_m}\right)  (\o, p)\\
&=  \frac{\partial^m}{\partial p^m}\Tc^m(g) (\o, p).
\end{align*}
\end{proof}
\noindent As proved in Theorem \ref{th:kernel description}, all the terms except $\D^m v_m$ from the decomposition \eqref{eq:decomposition of f} lie in the kernel of the transversal Radon transform. The next theorem shows that $v_m$ can be uniquely recovered from transversal Radon transform of $f$.
\begin{theorem} \label{thm: recovery of $v_m$ using TRT}
    The potential part $v_m$ of a symmetric $m$-tensor field $f$ can be recovered explicitly from the knowledge of its transversal Radon transform $\Tc^m f$. 
\end{theorem}
\begin{proof}
We know a symmetric $m$-tensor field $f$ can be decomposed as follows:
$$ f  = \sum_{i=0}^{m} \D^iv_i.$$
Recall from Theorem \ref{th:kernel description} that the part $\displaystyle \left(\sum_{i=0}^{m-1} \D^iv_i\right)$ of $f$ is in the kernel of $\Tc^m$ and hence we have 
\begin{align*}
\Tc^m f (\o, p) =  \Tc^m\left(\sum_{i=0}^{m} \D^iv_i\right)(\o, p) =  \Tc^m\left(\D^m v_m\right)(\o, p).
\end{align*}
The aim here is to recover $v_m$ from the knowledge of $\Tc^m f$. To achieve this, we use  the following relation obtained from Lemma \ref{lem: Radon of divergence}:
\begin{align*}
 \Rc (\delta^m \D^m v_m) (\o, p) = \frac{\partial^m}{\partial p^m}\Tc^m \left(\D^m v_m\right)(\o, p) = \frac{\partial^m}{\partial p^m}\Tc^m f(\o, p).
\end{align*}
This shows that one can find $\delta^m \D^m  v_m$ explicitly in terms of known data by using the inversion of the usual Radon transform. Finally, we apply Theorem \ref{th:solvability of elliptic system} to find $v_m \in H^{s+2m}_t(\Rb^n)$.
\end{proof}
\noindent Now, we prove that $v_{m - k}$ for $k = 1, \dots, m$ can be reconstructed using weighted transversal Radon transforms of order $k$. The proof is done using mathematical induction on $k$. In the following subsection, we prove it for $k = 1$. We start by proving a proposition that will be crucial to proceed further. 

\subsection{Recovery of \texorpdfstring{$v_{m-1}$}{vm-1} from \texorpdfstring{$\Tc^{w,1}_{\ell_1 \dots \ell_{n - 1}} f$}{Tcm ell1... elln-1 f}} 

\begin{proposition} \label{prop: for solenoidal fields}
If $v$ is a solenoidal symmetric $m$-tensor field, that is $\delta v = 0$, then $\delta^k d^\ell v = 0$ for $k > \ell > 0$. In particular, $\delta^\ell d^\ell v$ is solenoidal.
\end{proposition}
\begin{proof}
To prove this proposition, it is sufficient to show that $\delta^{\ell + 1}\D^\ell v = 0$. We start by proving the following:
\begin{equation} \label{eq: for solenoidal fields}
    \delta^\ell\D^\ell = \sum_{i = 0}^\ell c_i {(\overline{\Delta}})^{\ell - i} (\D \delta)^i,
\end{equation}
where $\overline{\Delta}$ represents the component-wise Laplacian and $c_i \in \FR$ for $i = 0, 1, \dots, \ell$ are constants. If the above equation is true and $v$ is solenoidal $(\delta v = 0)$, then we get $\delta^\ell\D^\ell v = c_0 {(\overline{\Delta}})^{\ell} v$. Further using $\overline{\Delta} \delta = \delta \overline{\Delta}$, we get $\delta^{\ell + 1}\D^\ell v = 0$. We prove equation \eqref{eq: for solenoidal fields} using mathematical induction on $\ell$. 
For $\ell = 1$, we have
\begin{align*}
(\delta \D v)_{i_1 \dots i_m} &= \delta \left\{\sigma(i_1 \dots i_{m + 1}) \frac{\partial v_{i_1 \dots i_m}}{\partial x_{i_{m + 1}}}\right\} \nonumber\\
&= \frac{1}{m + 1} \delta \left\{\frac{\partial v_{i_1 \dots i_m}}{\partial x_{i_{m + 1}}} + \sum_{j = 1}^m \frac{\partial v_{i_1 \dots i_{j - 1} i_{j + 1} \dots i_{m + 1}}}{\partial x_{i_j}}\right\}\nonumber\\
&= \frac{1}{m + 1} \left\{\frac{\partial^2 v_{i_1 \dots i_m}}{\partial x^2_{i_{m + 1}}} + \sum_{j = 1}^m \frac{\partial}{\partial x_{i_j}} \left(\frac{\partial v_{i_1 \dots i_{j - 1} i_{j + 1} \dots i_{m + 1}}}{\partial x_{i_{m + 1}}}\right)\right\}\nonumber\\
&= \frac{1}{m + 1} \left\{\Delta (v_{i_1 \dots i_m}) + \sum_{j = 1}^m \frac{\partial}{\partial x_{i_j}} {(\delta v)}_{i_1 \dots i_{j - 1} i_{j + 1} \dots i_m}\right\}\nonumber\\
&= \frac{1}{m + 1} \Delta (v_{i_1 \dots i_m}) +  \frac{m}{m + 1} (\D \delta v)_{i_1 \dots i_m}.
\end{align*}
Therefore, we have 
\begin{equation}\label{eq: commutativity of two operators}
\delta \D v = c_0 \overline{\Delta} v + c_1 \D \delta v, \mbox{  where  } c_0 = \frac{1}{m + 1} \mbox{ and } c_1 = \frac{m}{m + 1}.
\end{equation}
This implies that the claim \eqref{eq: for solenoidal fields} is true for $\ell = 1$. Now assume that the claim is true for $\ell = p - 1$ for some $p \geq 2$, that is, there exist constants $c_i$; $i = 0, \dots, p - 1$ such that
\begin{align*}
    \delta^{p - 1}\D^{p - 1}  = \sum_{i = 0}^{p - 1} c_i {(\overline{\Delta})}^{p - 1 - i} (\D \delta)^i.   
\end{align*}
Then for $\ell = p$, we have
\begin{align*}
\delta^p\D^p &= \delta (\delta^{p - 1}\D^{p - 1}) \D \\ 
&= \delta \left(\sum_{i = 0}^{p - 1} c_i {(\overline{\Delta}})^{p - 1 - i} (\D \delta)^i\right) \D , \quad \mbox{ using the induction hypothesis }\\
&= \sum_{i = 0}^{p - 1} c_i {(\overline{\Delta})}^{p - 1 - i} (\delta \D)^{i + 1} \\
&= \sum_{i = 0}^{p - 1} c_i {(\overline{\Delta})}^{p - 1 - i} (c_0 \overline{\Delta}  + c_1 \D \delta)^{i + 1}, \quad \mbox{using equation \eqref{eq: commutativity of two operators}}\\
&= \sum_{i = 0}^p \tilde{c_i} {(\overline{\Delta}})^{p - i} (\D \delta)^i,
\end{align*}
where $\tilde{c_i}$s are functions of $c_0, c_1, \dots, c_{p-1}$. This completes the proof of the proposition.
\end{proof}
\noindent We note that we can expand, $\delta^{m-1}d^{m-1}v_{m-1}= \sum_{i = 1}^{n - 1} \left< (\delta^{m - 1} \D^{m - 1} v_{m - 1}), \omega_i \right> \omega_i + \left<  (\delta^{m - 1} \D^{m - 1} v_{m - 1}), \omega \right> \omega$. Hence the componentwise Radon transform $\overline{\Rc}$ of $\delta^{m - 1} \D^{m - 1} v_{m - 1}$ can be written as follows:
\begin{align} \label{eq: componentwise Radon transform WTRT1}
    \overline{\Rc} (\delta^{m - 1} \D^{m - 1} v_{m - 1}) = \sum_{i = 1}^{n - 1} \left< \overline{\Rc} (\delta^{m - 1} \D^{m - 1} v_{m - 1}), \omega_i \right> \omega_i + \left< \overline{\Rc} (\delta^{m - 1} \D^{m - 1} v_{m - 1}), \omega \right> \omega
\end{align}
The last term in the above equation is the transversal Radon transform of $\delta^{m - 1} \D^{m - 1} v_{m - 1}$, i.e. $$\left< \overline{\Rc} (\delta^{m - 1} \D^{m - 1} v_{m - 1}), \omega \right> \omega =T^1{(\delta^{m - 1} \D^{m - 1} v_{m - 1})}.$$ Recall that $v_{m-1}$ is solenoidal. Thus, $\delta^{m - 1} \D^{m - 1} v_{m - 1}$ is solenoidal (using Proposition \ref{prop: for solenoidal fields}), and the last term in \eqref{eq: componentwise Radon transform WTRT1} is zero (Theorem \ref{th:kernel description}). Further, $\left< \overline{\Rc} (\delta^{m - 1} \D^{m - 1} v_{m - 1}), \omega_i \right>$ can be obtained from $\Tc^{w,1}_{\ell_1 \dots \ell_{n - 1}}$ for $\ell_i = 1$ and $\ell_j = 0$ for $j \neq i$. We have
\begin{equation*}
    \Tc^{w,1}_{\ell_1 \dots \ell_{n - 1}} f = \mathcal{T}^{m} \left(\left< x, \omega_i\right> f\right); \quad 1 \leq i \leq n - 1, \ell_i = 1 \mbox{ and } \ell_j = 0 \mbox{ for } j \neq i
\end{equation*}
Using Lemma \ref{lem: Radon of divergence}, we get
\begin{align}\label{eq:15}
 \frac{\partial^m}{\partial p^m} \Tc^{w,1}_{\ell_1 \dots \ell_{n - 1}} f = \frac{\partial^m}{\partial p^m} \mathcal{T} \left(\left< x, \omega_i\right> f\right) = \Rc(\delta^m \left(\left< x, \omega_i\right> f\right)).   
\end{align}
\noindent By direct calculation, one can verify that,$$\delta \left(\left< x, \omega_i\right> f\right) = \left< \omega_i,  f\right> + \left< x, \omega_i\right> \delta f.$$
\noindent Now, by an inductive argument, it is easy to show that for any $1\leq k\leq m$
\begin{align}\label{eq:16}
  \delta^k \left(\left< x, \omega_i\right> f\right) = k \left< \omega_i, \delta^{k - 1} f\right> + \left< x, \omega_i\right> \delta^k f.  
\end{align}
\noindent By letting $k=m$ in \eqref{eq:16} and substituting it in \eqref{eq:15} we get,
\begin{equation} \label{eq: relation between linear weighted transverse transform and Radon transform}
    \frac{\partial^m}{\partial p^m} \Tc^{w,1}_{\ell_1 \dots \ell_{n - 1}} f = \Rc \left\{m \left< \omega_i, \delta^{m - 1} f\right> + \left< x, \omega_i\right> \delta^m f\right\}.
\end{equation}
\noindent Using Proposition \ref{prop: for solenoidal fields} and the Decomposition \ref{Decomposition Result}, we have
$$\delta^{m - 1} f = \delta^{m - 1} \D^{m - 1} v_{m - 1} + \delta^{m - 1} \D^m v_m$$ and
$$\delta^m f = \delta^m \D^m v_m.$$
\noindent Substituting the above relations in equation \eqref{eq: relation between linear weighted transverse transform and Radon transform}, we get
\begin{align*}
\frac{\partial^m}{\partial p^m} \Tc^{w,1}_{\ell_1 \dots \ell_{n - 1}} f &= m \left< \omega_i, \overline{\Rc} \left(\delta^{m - 1} \D^{m - 1} v_{m - 1}\right)\right> + m \left< \omega_i, \overline{\Rc} \left(\delta^{m - 1} \D^m v_m\right)\right> + \Rc \left(\left< x, \omega_i \right> \delta^m \D^m v_m \right)\\
\implies \left< \omega_i, \overline{\Rc} \left(\delta^{m - 1} \D^{m - 1} v_{m - 1}\right)\right> &= \frac{1}{m} \left[\frac{\partial^m}{\partial p^m} \Tc^{w,1}_{\ell_1 \dots \ell_{n - 1}} f - m \left< \omega_i, \overline{\Rc} \left(\delta^{m - 1} \D^m v_m\right)\right> - \Rc \left(\left< x, \omega_i \right> \delta^m \D^m v_m \right)\right]
\end{align*}
\noindent Above equation is true for all $1 \leq i \leq n - 1$. Substituting the above equation in equation \eqref{eq: componentwise Radon transform WTRT1}, we get the componentwise Radon transform of $\delta^{m - 1} \D^{m - 1} v_{m - 1}$. Then using Radon inversion, $\delta^{m - 1} \D^{m - 1} v_{m - 1}$ can be obtained. Again, we use Theorem \ref{th:solvability of elliptic system} to find $v_{m-1} \in H^{s+2m-1}_t(\Rb^n)$, which completes the goal of this section.

\subsection{Recovery of \texorpdfstring{$v_{m-k}$}{vmk} from \texorpdfstring{$\Tc^{w,k}_{\ell_1 \dots \ell_{n - 1}} f$}{Tcm ell1... elln-1 f}} 
Finally, we assume $v_{m - k + 1}, v_{m - k + 2}, \dots, v_m$ are known for $2 \leq k \leq m$ and  show $\delta^{m - k} \D^{m - k} v_{m - k}$ can be obtained from the transforms $\Tc^{w,k}_{\ell_1 \dots \ell_{n - 1}} f$. Similar to equation \eqref{eq: componentwise Radon transform}, the componentwise Radon transform $\overline{\Rc}$ of $\delta^{m - k} \D^{m - k} v_{m - k}$ can be written as follows:
\begin{align}
    &\overline{\Rc} \left(\delta^{m - k} \D^{m - k} v_{m - k}\right)\nonumber\\ 
    &\quad = \sum_{\ell_1 + \dots + \ell_{n - 1} = k} \left< \overline{\Rc} \left(\delta^{m - k} \D^{m - k} v_{m - k}\right), \o_1^{\odot \ell_1} \odot \o_2^{\odot \ell_2} \odot \dots \odot \o_{n - 1}^{\odot \ell_{n - 1}} \right>\o_1^{\odot \ell_1} \odot \o_2^{\odot \ell_2} \odot \dots \odot \o_{n - 1}^{\odot \ell_{n - 1}}\nonumber\\
    &\qquad \qquad + \sum_{\ell_1 + \dots + \ell_n = k - 1} \left< \overline{\Rc} \left(\delta^{m - k} \D^{m - k} v_{m - k}\right), \o \odot \o_1^{\odot \ell_1} \odot \o_2^{\odot \ell_2} \odot \dots \odot \o_{n - 1}^{\odot \ell_{n - 1}} \odot \o^{\odot \ell_n} \right>\nonumber\\
    &\hspace{7cm} \times \o_1^{\odot \ell_1} \odot \o_2^{\odot \ell_2} \odot \dots \odot \o_{n - 1}^{\odot \ell_{n - 1}} \odot \o^{\odot \ell_n + 1}.
\end{align}
The terms in the second summation of the above equation are mixed and transversal Radon transforms of $\delta^{m - k} \D^{m - k} v_{m - k}$. From Proposition \ref{prop: for solenoidal fields}, we know that $\delta^{m - k} \D^{m - k} v_{m - k}$ is solenoidal and using Theorem \ref{th:kernel description}, we get that the second summation is zero. Further, we want to show that 
$$\left< \overline{\Rc} \left(\delta^{m - k} \D^{m - k} v_{m - k}\right), \o_1^{\odot \ell_1} \odot \o_2^{\odot \ell_2} \odot \dots \odot \o_{n - 1}^{\odot \ell_{n - 1}} \right>$$ 
can be written in terms of $W^k_{\ell_1 \dots \ell_{n - 1}} f$. We prove this for $\ell_1 = k$ and $\ell_j = 0$ for $j \neq 1$. That is, we show that $\left< \overline{\Rc} \left(\delta^{m - k} \D^{m - k} v_{m - k}\right), \o_1^{\odot k}\right>$ can be written in terms of $W^k_{\ell_1 \dots \ell_{n - 1}} f$ for $\ell_1 = k$ and $\ell_j = 0$ for $j \neq 1$. The proof for the other cases can be done in a similar manner.\\

\noindent For $g = \left<x, \o_1\right> ^k f$, we have
$$\Tc^{w,k}_{\ell_1 \dots \ell_{n - 1}} f = \mathcal{T} g$$ where $\ell_1 = k$ and $\ell_j = 0$ for $j \neq 1$. Using Lemma \ref{lem: Radon of divergence}, we get
\begin{equation} \label{eq: relation between weighted transverse transform and Radon transform}
    \frac{\partial^m}{\partial p^m} \Tc^{w,k}_{\ell_1 \dots \ell_{n - 1}} f = \frac{\partial^m}{\partial p^m} \mathcal{T} g = \Rc(\delta^m g).
\end{equation}
\noindent Now, we show that 
\begin{equation} \label{eq: delta operator induction}
    \delta^p g = \sum_{i = r}^{k - 1} \binom{k}{i} p (p - 1) \dots \left((p + 1) - (k - i)\right) \left<x, \o_1\right>^i \left<\o_1^{\odot k - i}, \delta^{p - (k - i)} f\right> + \left<x, \o_1\right>^k \delta^p f 
\end{equation}
where 
\[r = \begin{cases*}
    k - p; & $p < k$\\
    0; & $p \geq k$.
\end{cases*}\]
This is proved using mathematical induction on $p$. For $p = 1$, applying $\delta$ operator to $g = \left<x, \o_1\right> ^k f$ gives
\begin{equation*}
    \delta g = k \left<x, \o_1\right>^{k - 1} \left<\o_1, f\right> + \left<x, \o_1\right>^k \delta f.
\end{equation*}
Hence, equation \eqref{eq: delta operator induction} is true for $p = 1$. Assume that it is true for all $p < k$. Then for $p = k$, we have
\begin{align*}
    \delta^k g &= \delta (\delta^{k - 1} g)\\
    &= \delta \left[ \sum_{i = 1}^{k - 1} \binom{k}{i} \left((k - 1) (k - 2) \dots i\right) \left<x, \o_1\right>^i \left<\o_1^{\odot k - i}, \delta^{i - 1} f\right> + \left<x, \o_1\right>^k \delta^{k - 1} f\right]\\
    &= \sum_{i = 1}^{k - 1} \binom{k}{i} \left((k - 1) (k - 2) \dots i\right) \left(i \left<x, \o_1\right>^{i - 1} \left<\o_1^{\odot k + 1 - i}, \delta^{i - 1} f\right>\right)\\
    &\qquad + \sum_{i = 1}^{k - 1} \binom{k}{i} \left((k - 1) (k - 2) \dots i\right) \left(\left<x, \o_1\right>^i \left<\o_1^{\odot k - i}, \delta^i f\right>\right) + k \left<x, \o_1\right>^{k - 1} \left<\o_1,\delta^{k - 1} f\right> + \left<x, \o_1\right>^k \delta^k f\\
    &= \sum_{i = 0}^{k - 2} \binom{k}{i + 1} \left((k - 1) (k - 2) \dots (i + 1)\right) \left((i + 1) \left<x, \o_1\right>^{i} \left<\o_1^{\odot k - i}, \delta^{i} f\right>\right) \hspace{1.5cm} \mbox{(replacing $i$ with $i + 1$)}\\
    &\qquad + \sum_{i = 1}^{k - 1} \binom{k}{i} \left((k - 1) (k - 2) \dots i\right) \left(\left<x, \o_1\right>^i \left<\o_1^{\odot k - i}, \delta^i f\right>\right) + k \left<x, \o_1\right>^{k - 1} \left<\o_1,\delta^{k - 1} f\right> + \left<x, \o_1\right>^k \delta^k f\\
    &= \left(k (k - 1) \dots 1\right) \left<\o_1^{\odot k}, f\right> + \sum_{i = 1}^{k - 2} (k - 1) (k - 2) \dots (i + 1) \left[(i + 1)\binom{k}{i + 1} + i \binom{k}{i}\right] \left<x, \o_1\right>^{i} \left<\o_1^{\odot k - i}, \delta^{i} f\right>\\
    &\qquad + 2 k \left<x, \o_1\right>^{k - 1} \left<\o_1,\delta^{k - 1} f\right> + \left<x, \o_1\right>^k \delta^k f\\
    &= \sum_{i = 0}^{k - 1} \binom{k}{i} k (k - 1) \dots (i+1) \left<x, \o_1\right>^i \left<\o_1^{\odot k - i}, \delta^{i} f\right> + \left<x, \o_1\right>^k \delta^k f
\end{align*}
using the relation $$(i + 1)\binom{k}{i + 1} + i \binom{k}{i} = k \binom{k}{i}.$$

\noindent Repeating the above steps multiple times, equation \eqref{eq: delta operator induction} can be proved for $p \geq k$. Since $1 \leq k \leq m$, putting $p = m$ in equation \eqref{eq: delta operator induction}, we get

\begin{align*}
    \delta^m g 
    = \sum_{i = 0}^{k - 1} \binom{k}{i} m (m - 1) \dots \left((m + 1) - (k - i)\right) \left<x, \o_1\right>^i \left<\o_1^{\odot k - i}, \delta^{m - (k - i)} f\right> + \left<x, \o_1\right>^k \delta^m f.
\end{align*}
\noindent Further, using Proposition \ref{prop: for solenoidal fields} and the Decomposition \ref{Decomposition Result}, we have
\begin{align*}
\delta^m f &= \delta^m \D^m v_m\\
\delta^{m - 1} f &= \delta^{m - 1} \D^{m - 1} v_{m - 1} + \delta^{m - 1} \D^m v_m\\
&\vdots\\
\delta^{m - k} f &= \delta^{m - k} \D^{m - k} v_{m - k} + \delta^{m - k} \D^{m - k + 1} v_{m - k + 1} + \dots + \delta^{m - k} \D^{m} v_m
\end{align*}
Since $v_{m - k + 1}, v_{m - k + 2}, \dots, v_m$ are known, substituting the above values in equation \eqref{eq: relation between weighted transverse transform and Radon transform} gives an expression for $\left<\overline{\Rc}\left(\delta^{m - k} \D^{m - k} v_{m - k}\right), \o_1^{\odot k}\right>$ in terms of the known values. Similarly, each term of equation \eqref{eq: componentwise Radon transform} can be written in terms of $v_{m - k + 1}, v_{m - k + 2}, \dots, v_m$ and the known transforms. Then using Radon inversion, $\delta^{m - k} \D^{m - k} v_{m - k}$ can be recovered. Finally, $v_{m - k}$ can be obtained by using Theorem \ref{th:solvability of elliptic system}, and this completes the theorem.  
\section{Acknowledgements}\label{sec: acknowledge} RM was partially supported by SERB SRG grant No. SRG/2022/000947. CT was supported by the Prime Minister's Research Fellowship from the Government of India.
\bibliographystyle{amsplain}
\bibliography{reference}

\end{document}